\documentclass[11pt,a4paper, DIV=12]{scrartcl}
\usepackage[utf8]{inputenc}
\usepackage[T1]{fontenc}
\usepackage{amsmath}
\usepackage{amssymb}
\usepackage{amscd}
\usepackage{amsthm}
\usepackage{mathrsfs}
\usepackage{proof}
\usepackage{color}
\usepackage{enumerate}
\usepackage{hyperref}
\setkomafont{sectioning}{\bfseries}

\newtheorem{theorem}{Theorem}[section]
\newtheorem{coro}[theorem]{Corollary}
\newtheorem{lemma}[theorem]{Lemma}
\newtheorem{prop}[theorem]{Proposition}

\newtheorem*{theorem*}{Theorem}
\theoremstyle{definition}
\newtheorem{definition}[theorem]{Definition}

\newtheorem*{remark}{Remark}
\newtheorem{remarks}[theorem]{Remarks}

\newcommand{\en}{{\cal E}}

\newcommand{\vol}{\mathrm{vol}}

\newcommand{\diam}{\mathrm{diam}}

\newcommand{\cC}{\mathcal{C}}
\newcommand{\cH}{\mathcal{H}}
\newcommand{\cF}{\mathcal{F}}
\newcommand{\D}{\mathrm{D}}
\newcommand{\En}{\mathscr{E}}
\newcommand{\F}{\mathscr{F}}
\newcommand{\Xp}{X_0}
\newcommand{\Ep}{\en_0}
\newcommand{\mep}{\m_0}
\newcommand{\m}{\mathfrak{m}}
\newcommand{\dm}{\mathsf{d}}
\newcommand{\fracc}[2]{{\scriptstyle{\frac{#1}{#2}}}}
\newcommand{\che}[1]{\mathsf{Ch}_*(#1)}
\newcommand{\NN}{\mathbb{N}}

\newcommand{\RR}{\mathbb{R}}

\begin{document}

\title{A new uncertainty principle at low energies}

\author{D.~Lenz\footnote{ Mathematisches Institut, Friedrich Schiller
Universit{\"a}t Jena, 07743 Jena, Germany, daniel.lenz@uni-jena.de},
P.~Stollmann\footnote{Fakult\"{a}t f\"{u}r Mathematik,  Technische
Universit\"{a}t Chemnitz, D-09107 Chemnitz, Germany,
stollman@math.tu-chemnitz.de }
 }

\maketitle

\begin{abstract}
We present a new method to prove uncertainty principles for functions of low energy in a very general set-up. This leads to results for metric measure spaces, Dirichlet spaces and graphs under rather weak assumptions.
\end{abstract}


\section*{Introduction}

We consider quite general energy forms and study how functions with small energies are spread out in space. By  spread out we mean here  that the norm of the function can be controlled by the norm of its restriction to subsets.  This issue is often discussed under the heading uncertainty principle. It has attracted attention for various reasons in the last decades. To give a feeling of what such uncertainty estimates may look like and how they can be used, we first mention \textbf{uncertainty principles for spectral projectors}, taking the form
\begin{equation} \label{eq:intro1}
P_E V P_E \geq \kappa P_E ,
\end{equation}
where, in a Hilbert space $\mathcal{H}$, $P_E=\boldsymbol{1}_{[0,E]}(H)$ is the spectral projector associated 
to some selfadjoint operator $H\ge 0$ onto energies below $E \geq \min \sigma (H)$ and $V\ge 0$ is a bounded selfadjoint operator. More specifically, the reader should think of $\mathcal{H}=L^2(X)$, where $X$ is a discrete or continuum configuration space, $H=-\Delta$ a Laplace operator and $V=1_Y$ the indicator function of a set that is sufficiently spread out in $X$ (we have not yet properly introduced the previous notions, see below for details). Note that \eqref{eq:intro1} now takes the form 
\begin{equation}
 \label{eq:intro2}
\| 1_Y f\|^2 \geq \kappa \| f\|^2\mbox{   if  }f\in Ran(P_E) .
\end{equation}
This states that the norm of $f$ on all of $X$ is controlled by the norm of $f$ on some subset $Y$. Of course, the fundamental question  is whether such an uncertainty inequality holds for given $H$, some energy cut-off $E$ and non-trivial subsets $Y\subset X$.
\begin{itemize}
\item Our original motivation to study the topic came from the field of random Schrödinger operators, where uncertainty estimates like \eqref{eq:intro2} above help to prove Wegner estimates in particular for models in which the random perturbation does not cover the whole configuration space, see \cite{BLS,CHK-03,CHK-07,K-13,MRM-22,NTTV-18,RMV-13,stolzman} and the literature cited there. 
\item A second field of application is control theory. Here, uncertainty relations are used as an input to the Lebeau-Robbiano strategy \cite{LebeauR-95}. The general picture is that an uncertainty relation implies an observability estimate, which by Douglas' lemma \cite{Douglas-66}, is equivalent to null-controllability of a certain linear control problem. We mention  \cite{LebeauZ-98,JerisonL-99,Miller-10,TenenbaumT-11,WangZ-17,BeauchardP-18,NakicTTV-20,GallaunST-20} for the subsequent development of the method. While the original Lebeau-Robbiano strategy requires uncertainty estimates ``at all energies'' it has recently been realized that uncertainty estimates at low energies imply a somewhat weaker form of controllability, see \cite{BST,EgidiGST-24,MuenchSST}.
\item In the rapidly developping field of signal processing on graphs, the set-up would typically be a finite weighted graph $X$, $H$ a suitable discrete Laplacian and $Y$ a subset of nodes. A function $f\in Ran(P_E)$ would be called a band-limited graph signal and inequality \eqref{eq:intro2} readily implies that such a function is uniquely determined its values on $Y$. For the corresponding theory of sampling we refer to the fundamental paper \cite{pesenson-2} by I.Z. Pesenson, which is frequently cited in the graph signal literature, see, e.g. \cite{ago,al,cvsk,ptgv,snfov,tbl}.
\end{itemize}

We present a new and rather simple approach to prove a more general version of such uncertainty estimates.  Our method, as captured  in Theorem \ref{theorem:main-technical} gives the following \textbf{uncertainty principle at low energy}, where we suppress certain explicit constants to underline the analogy to \eqref{eq:intro2} above; for the notions used, the reader has to look at the subsequent section:
\medskip

\noindent\textbf{Theorem 1.5 - slightly rephrased.}
\textit{Let $\mathcal{E}$ be an energy functional on  $L^2(X,\m)$. Assume that there exists a  $(\lambda,M)$-decomposition $\cC$,  and let $\varrho>0$.  Then, there exist $E,\kappa>0$ so that  for any  measurable subset $Y$  of $X$ that is $\varrho$-thick with respect to $\cC$ and any $f\in L^2(X,\m)$:}
\begin{equation}\label{UP}\tag{UP}
\mathcal{E} (f) \leq E \|f\|^2
\quad\Longrightarrow\quad
\|f\|^2 \leq \frac{1}{\kappa} \|f 1_Y\|^2.
\end{equation}

First observe that if the energy functional is the quadratic form of a selfadjoint operator $H$, then \eqref{UP} implies \eqref{eq:intro2} with the same constants. As shown in \cite{BST}, these two versions of an uncertainty estimate are actually equivalent. However, for the general set-up we consider here, the energy form is not required to be quadratic.

We note that our result is rather uniform in $Y$ in that the constants only depend on $Y$ through the  parameter $\varrho$, which can be seen as its density or thickness.  We also note that this dependence on $\varrho$ is unavoidable as the case of a compact space $X$ and $f = 1$ shows. 

Our approach can be seen as a ``local to global'' strategy: it turns local Poincaré inequalities to a quantitative uncertainty principle of the form \eqref{UP}. At the heart of the matter is the rather elementary observation that a global energy estimate implies quite uniform local estimates through  Naißigkait's lemma, \ref{nice}, below. This general, abstract part is given in the next section.

Section 2 is devoted to metric measure spaces. Here, we consider the Cheeger energy as energy functional. The main work to be done is to identify the necessary conditions that allow one to construct what is called a $(\lambda,M)$-decomposition of the underlying space $X$. It should be thought of as breaking up $X$ into pieces of finite measure with two important properties: the covering is locally finite, with at most $M$ of the pieces overlapping. On each of the pieces a Poincaré inequality with constant $\lambda$ holds. We will achieve that goal in terms of a family of open balls with the same radius $R$ provided the underlying metric measure space satisfies a weak volume comparison property, substantially weaker than the volume doubling condition which is used quite often in the metric measure space literature. We also have to assume a uniform Poincaré inequality on balls of radius $R$. We should point out, that this set-up contains cases for which the energy is not quadratic, hence beyond the techniques available in the literature so far.

In Section 3 we show how to treat quasi-regular Dirichlet forms with the abstract method. This leads to conditions pretty much like the ones in the metric measure space setting.

Section 4 deals with weighted graphs, a case well-studied in different generality, we already mentioned \cite{pesenson-2} above. We discuss in detail the resulting estimates and compare them to inequalities obtained earlier, with different methods.

\section{The abstract core}\label{section:technical-core}
In this section we present a general result capturing our approach.  We proceed in two steps. In the first step  we  consider the case of a finite measure space. Building on this we then treat the general case by suitable decompositions.  The considerations in this section neither assume that the functional in question is a quadratic form nor that the underlying space carries  a topology. 

\medskip

Throughout we denote the characteristic function of a set $K$ by $1_K$. 

An \textit{energy functional} $\mathcal{E}$ on the Hilbert space $\cH$ is a map  $\mathcal{E} : \mathcal{F} \longrightarrow [0,\infty)$, where $\mathcal{F}$ is a subspace of $\cH$ referred to as  the \textit{domain} of $\mathcal{E}$. Whenever convenient, we extend the functional to all of $\cH$ by setting it equal to $\infty$ on $\cH\setminus \cF$.

Of course, our notation is inspired by Dirichlet form theory but for our purposes here, we do not require the functional to be quadratic, let alone closed or densely defined. For general facts on closed forms and the associated operators, see \cite{kato,reed-simon}.

Let $(\Xp,\mep)$ be a finite measure space. The energy functional $\Ep$ on $L^2 (\Xp,\mep)$  is said to satisfy a \textit{Poincar\'{e} inequality} (with parameter $\lambda >0$)  if any $f$ in the domain of $\Ep$ satisfies the inequality 
\begin{equation}\label{PI}\tag{PI}
 \lambda \|f -  A(f) \underline{1}\|^2 \leq \Ep(f).
\end{equation}
Here, the  average $A  (f) $ of $f$ given by $$A (f) := \frac{1}{\mep(\Xp)} \int_{\Xp} f d\mep
$$
and $\underline{1}$ denotes the constant function with value $1$. 
Note that  $\Ep\geq 0$ must hold if $\Ep$ satisfies a Poincar\'{e} inequality.

\begin{remark}[The Poincar\'{e} inequality as a spectral bound.] If $\Ep$ is a quadratic form with $\Ep (1) =0$ then a Poincar\'{e} inequality with parameter $\lambda$ is valid  if and only if 
$$\lambda \leq \lambda_1 :=\inf\{\frac{\Ep (g)}{\|g\|} : g\neq 0, \langle g,1\rangle = 1\}$$
holds (as in this case  $\Ep (f) = \Ep (f - A(f)\underline{1})$ holds). So, if $\Ep$ is a closed quadratic form associated with a selfadjoint operator $H_0$ and $0$ is an eigenvalue of $H_0$ with eigenspace given by the constant functions, then \eqref{PI} means that  the spectrum   of $H_0$ with $0$ removed is contained in $[\lambda,\infty)$. 
\end{remark}

By definition, the validity of a Poincar\'{e} inequality implies that functions with small values of $\Ep$  are essentially constant (and equal to its average). This allows one to show that a function with small value of $\Ep$  is governed by its values on a subset. Specifically, the following statement holds.

\begin{prop}[How small energy functions are governed  by subsets: finite measure case]  Let $(\Xp,\mep)$  be a finite measure space and  let $\Ep\geq 0$ be a closed form satisfying a Poincar\'{e} inequality with parameter $\lambda$. Let $Y$ be a measurable subset of $\Xp$  with positive measure. Let  $0< \varrho \leq \frac{\mep(Y)}{\mep(\Xp)}$ be given.  Then, for any $f\in L^2(\Xp,\mep)$  
$$\Ep(f) \leq \frac{\lambda \varrho}{4}  \|f\|^2
\quad\Longrightarrow \quad
\|f\|^2 \leq \frac{5}{\varrho} \|f 1_Y\|^2.$$
\end{prop}
\begin{proof} Without loss of generality we can assume $\|f\|=1$. We can decompose 
$f = c \underline{1} + \psi$ with $\psi \perp \underline{1}$ and $c = A(f)$. From the Poincar\'{e} inequality, \eqref{PI}, we find 
$$\|\psi\|^2 \leq \frac{1}{4} \varrho.$$
This gives 
$$c^2 \mep(X)  =\|c \underline{1}\|^2 = \|f\|^2 - \|\psi\|^2 \geq  1 - \frac{\varrho}{4}$$
and
$$ \|c 1_Y\|^2 = c^2 \mep(Y) \geq c^2 \mep(X) \varrho \geq  \varrho \left(1- \frac{\varrho}{4}\right)\geq \varrho \frac{3}{4}.$$
Altogether we then find  for 
$f 1_Y = c 1_Y + \psi 1_Y$ the estimate
$$\|f 1_Y\|^2  \geq ( \| c 1_Y\| - \|\psi\|)^2 \geq \left( \sqrt{\varrho} \left( 
\frac{\sqrt{3}}{{2}} - \frac{1}{2}\right)\right)^2 \geq \varrho \frac{1}{5}.$$
This finishes the proof. 
\end{proof}

\begin{remark}[Optimality of the constants] We comment on the constants $\lambda$  appearing in the assumption and the constant $\varrho$ appearing in the conclusion of the proposition. Given the generality of the statement, these constants are essentially optimal. To see this we consider the operator $L=- \frac{d^2}{d^x}$ with Neumann boundary conditions on $[0,2\pi]$.  Then, $L$ has pure point spectrum with eigenvalues $n^2,n=0,1,2,\ldots $ and associated eigenfunctions $\cos(n\cdot)$.

Now, the constant functions are eigenfunctions (to the eigenvalue $0$) and satisfy the assumption of the proposition. As these functions satisfy $\|f\|^2 = \frac{1}{\varrho} \|f 1_Y\|^2$ for $\varrho = \mep(Y)/\mep(X)$ we infer that the factor $\frac{1}{\varrho}$ in the conclusion  can not be avoided. 

Also, the  eigenfunctions for eigenvalues larger or equal to $\lambda =1$ have a zero within $[0,2\pi]$. Hence, chosing $Y$ as  a sufficiently small neighborhood of such a zero we infer that such eigenfunctions will not satisfy the conclusion of the statement. This shows that the factor $\lambda$ in the assumption of the proposition is necessary.  
\end{remark}

We now want to prove an analogue for not necessarily finite measure spaces. We will achieve this by decomposing our space into finite measure subspaces, use the above reasoning on each part of the decomposition and then combine the estimates. In order to combine the estimates from the different parts of the decomposition we use the following lemma (that may also be of interest in other situations). In the sequel, we call a set countable if it is finite or countably infinite.

\begin{lemma}[Naißigkait's Lemma]\label{nice}
Let $N$ be countable. 
Let  $a_k,b_k\geq 0$, $k\in N$, be given with  $\sum_{k} b_k\leq \sum_{k} a_k< \infty$. Then, 
$$(1-c)\sum_k a_k \leq  \sum_{k : a_k \geq c b_k} a_k$$
holds for any $c\geq 0$. 
\end{lemma}
\begin{proof} From $\sum_k b_k \leq \sum_k a_k$ we easily find 
$$\sum_{k : a_k < c b_k} a_k  \leq  \sum_{k} c b_k \leq c \sum_k a_k$$
and the desired statement follows. 
\end{proof}
\begin{remark}
 One way to view the preceding inequality is the following: Evidently
 $$
 b_k\le a_k\quad\Longrightarrow\quad \sum_k b_k \leq \sum_k a_k
 $$
 and the reverse implication is false. However, the slightly weaker (for $c\in [0,1)$) inequality $cb_k\le a_k$ holds for a substantial part of the terms, substantial measured in terms of the total sum of the corresponding terms $a_k$. We call it Naißigkait's Lemma, see \cite{PMFO}.
\end{remark}

We now turn to our notion of decomposition. For the rest of this section, let $(X,\m)$ be a measure space. 

\begin{definition}[Decomposition]\label{decompose}
Let $\mathcal{E}\geq 0$ be an energy functional on $L^2(X,\m)$.  Let $\lambda,M>0$ be given. A set $\cC$  of measurable subsets of $X$ with finite measure  is called a $(\lambda,M)$-\textit{decomposition of} $(X,\m)$ (\emph{with respect to} $\mathcal{E}$) if the following conditions hold:

\begin{itemize}
\item[(D1)] $1\leq \sum_{c\in\cC} 1_{C} \leq M$.

\item[(D2)] For any $C\in \cC$ there exists a form $\mathcal{E}_C$ on $L^2 (C,\m)$, such that for any $f$ in the domain of $\mathcal{E}$ the restriction $f|_C$ of $f$  to $C$ belongs to the domain of $\mathcal{E}_C$ and 
$$\frac{1}{M} \sum_C \mathcal{E}_C (f|_C) \leq \mathcal{E} (f).$$ 
\item[(D3)] For any $C\in \cC$, $\mathcal{E}_C$  satisfies a Poincar\'{e} inequality with parameter $\lambda$.
\end{itemize}
\end{definition}

Clearly, (D1) above means that $\cC$ is a locally finite covering of $X$ that is compatible with the form $\en$ by (D2), and the finite measure sets of this covering satisfy a uniform Poincaré inequality by (D3). Next we present our notion of relative denseness of a set, capturing the idea that it is spread out in space, in fact spread out with respect to a covering. 

\begin{definition}[Thick subsets]  Let $\cC$ be a covering of $X$ consisting  of subsets with finite measure. Let $\varrho>0$ be given. Then,  a measurable subset $Y$ of $X$ is called $\varrho$-\emph{thick with respect to} $\cC$ if
$$\m(C \cap Y) \geq \varrho \m(C)$$
holds for each $C\in \cC$. 
\end{definition}

Now, we have gathered the necessary ingredients to state and prove our main abstract result. 

\begin{theorem}[How small energy functions are governed by subsets: general case] \label{theorem:main-technical} Let $\mathcal{E}$ be an energy functional on  $L^2(X,\m)$. Assume that there exists a  $(\lambda,M)$-decomposition $\cC$,  and let $Y$ be a measurable subset of $X$ that is $\varrho$-thick with respect to $\cC$. Then, for any $f\in L^2(X,\m)$:
$$\mathcal{E} (f) \leq \frac{\lambda \varrho}{8 M} \|f\|^2
\quad\Longrightarrow\quad
\|f\|^2 \leq \frac{10 M}{\varrho} \|f 1_Y\|^2.$$
\end{theorem}
\begin{proof} Let $f\in L^2(X,\m)$. By property (D1) of $\cC$ we know that
$$
N:=\{C\in\cC : f|_C\not=0\}
$$
is countable.  By a slight abuse of notation we write $\|\cdot\|$ for the norm on any $L^2 (C,\m)$, $C\in N$ and for ease of notation, we use  $f_k$ for the restriction of $f$ to $k\in N$.   From the assumption on $f$ and the 
 defining properties of the decomposition we find
$$\frac{1}{M} \sum_k \mathcal{E}_k (f_k) \leq \mathcal{E} (f) \leq \frac{\lambda \varrho}{8M}  \sum_k \|f_k\|^2.$$
This implies
$$\sum_k \mathcal{E}_k (f_k) \leq \sum_k \frac{\lambda \varrho}{8} \|f_k\|^2.$$

We are going to  apply  Naißigkeit's Lemma with $b_k = \mathcal{E}_k (f_k)$ and $a_k = \frac{\lambda \varrho}{8} \|f_k\|^2$ and $c = \frac{1}{2}$. 
Specifically, we set 
$$S:=\{k\in N : \frac{1}{2} \mathcal{E}_k (f_k)  \leq \frac{\lambda \varrho}{8} \|f_k\|^2\}.$$
Then, we find from the Naißigkeit's Lemma 
$$\sum_{k\in S}  \|f_k\|^2\geq \frac{1}{2} \sum_k  \|f_k\|^2.$$
Moreover,  by its very definition we have for each $k\in S$ the inequality  $\mathcal{E}_k (f_k)\leq \frac{\lambda \varrho}{4 } \|f_k\|^2$. From the above proposition on the finite measure case we the find for any $k\in S$ the estimate
$$ \|f_k\|^2 \leq  \frac{5}{\varrho} \|f_k 1_Y\|^2.$$
Putting this together we arrive at
$$
\|f 1_Y\|^2 M \geq  \sum_k \|f_k 1_Y\|^2\\
\geq  \sum_{k\in S} \|f_k 1_Y\|^2\\
\geq \sum_{k\in S} \frac{\varrho}{5} \|f_k\|^2\\
\geq  \frac{\varrho}{10} \sum_k \|f_k\|^2 \\
\geq  \frac{\varrho}{10} \|f\|^2.$$
This gives the desired statement. 
\end{proof}
We should remark that our reasoning is in parts  somewhat reminiscent of the  arguments from \cite{pesenson}, where a local to global result for Poincaré inequality is presented for Dirichlet spaces. However, the set-up in the former paper is more restricted and the results there goe in a different direction.

\section{Application to metric measure spaces}\label{section:application-metric-measure-space}
In this section we apply our abstract result to study the Cheeger energy on  metric measure spaces. Such spaces have attracted a lot of attention in the last decades, see \cite{cheeger,heinonen,ags-2005} and the literature cited there, and in many works the authors assume volume doubling and local Poincaré inequalities. We show that under these quite natural assumptions an uncertainty principle holds. 

In fact, a much weaker (than doubling) volume comparison property will allow us to provide a decomposition, the Poincaré property being assumed, again, in a very weak form.

In order to include fairly general classes of spaces, we adopt the general set-up from \cite{ags-2014}. We always assume in this section that  $(X,\dm,\m)$ is an extended metric measure space, i.e. $\dm$ is an extended metric on $X$ that carries a topology making it a Polish extended space in the sense of Ambrosio, Gigli and Savar\'{e}, \cite{ags-2014}; moreover $\m$ is assumed to be a $\sigma$-finite measure on $X$ such that metric balls of positive radius  have finite positive measure. 
   We denote the open ball around $p\in X$ with radius $R$ by $U(p,R)$.  
As is well-kown any metric measure space with volume doubling allows for a locally finite covering by balls. To this end, the following much weaker property suffices: 

\begin{definition}
$(X,\dm,\m)$ \emph{satisfies the volume comparison property} \eqref{VCR} \emph{at scale} $R$ with constant $\gamma$ if:
\begin{equation}\label{VCR}\tag{VC$[\gamma,R]$}
 \m(U(p,4R)) \leq \gamma \m(U(p,\fracc{1}{2}R))\mbox{  for all }p\in X
\end{equation}
\end{definition}

\begin{prop}[Uniform covering]\label{uniformcov} Let $(X,\dm,\m)$ be a metric measure space satisfying \eqref{VCR} with $R>0$ and constant $\gamma$. Then we can find a countable $N\subset X$ with 
\begin{equation}\label{cov}\tag{cov}
 X = \bigcup_{p\in N} U(p,R)
\end{equation}
and, for any $x\in X$, 
 \begin{equation}\label{loc-fin}\tag{LF}
 \#\{p\in N: x\in U(p,R)\}\le \gamma
 \end{equation} 
\end{prop}
\begin{proof} By Zorn's lemma we can find a maximal set of disjoint balls of radius  $R/2$. As any of these balls has  positive measure and $\m$ is $\sigma$--finite, the set  is countable. By maximality any ball of radius $R/2$ intersects one of the $U(p,\fracc{1}{2}R)$. Thus $X = \bigcup_{p\in N} U(p,R)$ holds with a countable $N$. 

It remains to show \eqref{loc-fin}. Let $x\in X$, set 
$$L:=\{p\in N :  x\in U(p,R)\}$$
and let $p\in L$. Then,  the triangle inequality gives 
$$U(p,\fracc{1}{2}R)\subset U(x,\fracc{3}{2} R) \subset U(p,4R) .$$
holds.  In particular, the volume comparison property \eqref{VCR} implies that the following estimate holds for any  $p\in L$:
$$\m(U(x,\fracc{3}{2} R))  \leq \m(U(p,4R))\leq   \gamma  \m(U(p,\fracc{1}{2}R))$$

By disjointness of the balls $U(p,\fracc{1}{2}R)$, $k\in N$,   we furthermore infer
$$\sum_{p\in L} \m(U(p, \fracc{1}{2}R)) \leq \m(U(x,\fracc{3}{2} R)).$$
Putting this together we arrive at
$$\# L \m(U(x,\fracc{3}{2} R)) = \sum_{p\in L} \m(U(x,\fracc{3}{2} R))\leq  \sum_{p\in L} \gamma \m(U(p, \fracc{1}{2}R)) \leq \gamma  \m(U(x,\fracc{3}{2} R)).$$
This gives the desired estimate on $\# L$.  
\end{proof}
\begin{remarks}
\begin{itemize}
     \item[(1)] If the metric measure space is \textit{volume doubling}, i.e., if for all $R>0$ there is $\kappa>0$ such that 
$$\m(U(p,r)) \leq 2^\kappa \m(U(p,\fracc{1}{2}r))$$
holds for all $p\in X$ and $r<R$, then it satisfies VC$[\gamma,R]$ for all $R>0$ for suitable $\gamma$.
See \cite{pesenson} for an estimate in the volume doubling case that is slightly weaker.
\item[(2)] \eqref{VCR} is much weaker than volume doubling. It is true, provided the measures of balls of radius $\fracc{1}{2}R$  and those of radius $R$ have comparable volume. This is evidently true for all regular trees at all scales $R$ (with $R$-dependent $\gamma$, of course).
\item[(3)] The local finiteness condition \eqref{loc-fin} can be strengthend to obtain the following uniform bound: for all $q\in N$,
$$
\#\{ p\in N:U(q,R)\cap U(p,R)\not=\emptyset\}\le \gamma^\sharp,$$
using the stronger volume comparison assumption
\begin{equation}\label{VCRs}\tag{VC$^\sharp$[$R$]}
 \m(U(p,5R)) \leq \gamma^\sharp \m(U(p,\fracc{1}{2}R))\mbox{  for all }p\in X 
\end{equation}
with essentially the same proof as above.

 In particular, under these circumstances we can decompose $N$ into at most $k=\lfloor \gamma^\sharp\rfloor$ sets $N_l$, $l=1, \ldots, k$, such that
 $U(p,R)\cap U(q,R)=\emptyset$ for $l=1, \ldots, k$ and $p,q\in N_l$, $p\not= q$.
\item[(4)] The countability of $N$ is not relevant for certain applications and so we could omit the $\sigma$--finiteness of the measure from the assumption without changing the estimate \eqref{loc-fin}.
\end{itemize}
\end{remarks}

In the set-up of this section, a function $g\in L^2(X,\m)$ is called a \emph{relaxed slope of $f\in L^2(X,\m)$}, if there exists $\tilde{g}\in  L^2(X,\m)$ and a sequence of Lipschitz functions $f_n\in L^2(X,\m)$ such that 
\begin{itemize}
	\item $f_n\to f$ in $L^2(X,\m)$ and $|\mathrm{D}f_n|\to \tilde{g}$ weakly in $L^2(X,\m)$
	\item $\tilde{g}\leq g$,
\end{itemize}
where
$$
|\mathrm{D} h|(x):=\limsup_{y\to x}\frac{|h(x)-h(y)|}{\dm(x,y)}
$$
denotes the local Lipschitz constant of a local Lipschitz function $h:X\to\RR$, which by convention is put $0$, if $x$ is isolated. Such a $g$ is called the \emph{minimal relaxed slope of $f$}, if its $L^2$-norm is minimal amongst all relaxed slopes of $f$, and then one sets $|\D f|_*:=g$. We refer the reader to \cite{ags-2014} for more details, in particular for equivalent definitions of $|\D f|_*$. \\

In this generality, we follow \cite{ags-2014} and define the \emph{Cheeger energy functional} on $L^2(X,\m)$ by
\begin{equation}\label{cheegerenergy}\tag{C}
 \che{f}:=\frac{1}{2}\int_X|\D f|_*^2d\m ,
\end{equation}
set equal to $\infty$, if $f$ has no relaxed slope.
\begin{remarks}\label{remarksvolumedoubling}
 \begin{itemize}
  \item[(1)] The energy functional $\che{\cdot}$ is convex and lower semicontinuous.
  \item[(2)]  the domain of the Cheeger functional
  $$
  W^{1,2}_*(X,\dm,\m)=\{ f\in L^2(X,\m): \che{f}<\infty\}
  $$
  is a Banach space endowed with the norm  
  $$
  \| f\|_{W^{1,2}_*}:=\left(\| f\|^2+ \che{f}\right)^\frac12.$$
  \item[(3)] Under a rather weak  condition on the measure, (4.5) in \cite{ags-2014}, $W^{1,2}_*(X,\dm,\m)$  is dense in $L^2(X,\m)$.
 \end{itemize}
\end{remarks}
As a preparation for the main result of this section, we use the locality of the Cheeger energy to deduce:

\begin{prop}\label{propdecompose} Let $N\subset X$ satisfy \eqref{cov} and \eqref{loc-fin} above. Then  $\cC:=\{ U(p,R): p\in N\}$ and 
$$
\en_C:=\frac{1}{2}\int_C|\D g|_*^2d\m, 
$$
for $C\in\cC$, $g\in L^2(C,\m)$ satisfy the  properties \emph{(D1)} and \emph{(D2)} of Definition \ref{decompose}
 above.
\end{prop}
\begin{proof}
 Clearly, $C$, equipped with the induced metric and measure is an extended metric measure space itself and it is in this way that the integral in the definition of $\en_C$ above is to be understood. We now use subscript $C$ to denote the notions in this subspace and note the following evident facts:
 For $h\in L^2(X,\m)$ Lipshitz,
 $$
 |\D(h|_C)|_C\le |\D(h)||_C .
 $$
 Therefore, for any relaxed slope $g\in L^2(X,\m)$ of a given $f\in L^2(X,\m)$, $g|C$ is a relaxed slope of $f|C$, and so the minimal relaxed slope $|\D (f|_C)|_{C,*}$ calculated in the subspace $C$ satisfies: 
 $$
 |\D (f|_C)|_{C,*}\le |\D f|_*1_C,
 $$
 which gives 
 \begin{align*}
  \sum_C\en_C(f|_C)&=\frac12\sum_C\int_C|\D (f|_C)|_{C,*}^2d\m\\
  &\le \frac12\sum_C\int_X |\D f|_*^2 1_Cd\m\\
  &\le k \frac12\int_X |\D f|_*^2 d\m
 \end{align*}
\end{proof}

Now everything is prepared for our main result on metric measure spaces. 
We denote the average of $f$ on the ball $B$ by $A_B (f)$ i.e. 
$$A_B (f) = \frac{1}{\m(B)} \int_B f dm.$$

We say that a \emph{Poincaré inequality with constant} $\lambda$ at scale $R$ holds, if for every ball $B=U(p,R)\subset X$ and every $f\in L^2(B,\m)$
\begin{equation}\label{pilR}\tag{PI$[\lambda,R]$}
\lambda\int_B|f-A_B (f)|^2d\m\le \frac12\int_B|\D f|_*^2d\m .
\end{equation}

\begin{theorem}\label{mms}Assume that $R>0$ and $(X,\dm,\m)$ as above satisfies the volume comparison property \eqref{VCR} with constant $\gamma$ and the Poincaré inequality \eqref{pilR}.
Let $\varrho >0$. Then, for all $f\in L^2 (X,\m)$, 
$$
\che{f}\le \frac{\lambda\varrho}{\lfloor \gamma\rfloor}\| f\|^2
\quad\Longrightarrow\quad \| f\|^2\le\frac{10 \lfloor \gamma\rfloor}{\varrho}\| f 1_Y\|^2,
$$
whenever $Y\subset X$ is measurable and satisfies $\m(Y\cap B)\geq \varrho \m(B)$ for any ball  $B$ of  radius $R$.
\end{theorem}
\begin{proof}
 Using \eqref{VCR} and Proposition \ref{uniformcov} we find a countable $N\subset X$ such that 
 $$\cC=\{ U(p,R): p\in N\}$$ 
 as in Proposition \ref{propdecompose} satisfies (D1) and (D2) of Definition \ref{decompose}. Due to \eqref{pilR}, (D3) holds as well. This means that $\cC$ is a $(\lambda,M)$-decomposition with $M=\lfloor \gamma\rfloor$, so we can apply Theorem \ref{theorem:main-technical} giving the claim.
\end{proof}
\begin{remarks}
 \begin{itemize}
  \item[(a)] We thus obtain an uncertainty principle for functions with low Cheeger energy of the form \eqref{UP} with full control of the constants, provided the metric measure spaces satisfies
  a Poincaré inequality for balls and is volume doubling, by combining the preceding Theorem with (1) of Remarks \ref{remarksvolumedoubling}.
  \item[(b)] For many applications, in the Poincaré inequality for balls a specific dependence on the radius is supposed to hold. This is not relevant for our purposes here; it would simply mean that $\lambda$ in the estimate of Theorem \ref{theorem:main-technical} can be expressed in terms of $R$.
  \item[(c)] Consider weak type $(2,2)$ Poincaré inequalities of the form
	
  \begin{equation}\label{piw}\tag{weak-PI}
\lambda\int_B|f-A_B (f)|^2d\ m\le \frac12\int_{\Lambda B}|\D f|_*^2d\ m ,
\end{equation}
where $\Lambda >1$ and, for $B=U(p,R)$ one sets $\Lambda B=U(p,\Lambda R)$ as usual, disregarding the slight abuse of notation, since neither $p\in X$ nor $R$ are determined by the set $B$. 

With a proof just like in Proposition \ref{propdecompose}, a suitable volume comparison assumption weaker than volume doubling implies a uniform local bound on the number of balls of radius $\Lambda R$ meeting each fixed $x\in X$. This will give an estimate as in Theorem \ref{theorem:main-technical} above, even though the set-up is not, strictly speaking, covered by our abstract result. We refrained from including this slightly more general situation in view of notational simplicity.  
 \end{itemize}
\end{remarks}
We end this section with a direct consequence of Theorem 2.6 which is very much in the spirit of  \cite{BST}, where uncertainty principles are related with positivity of Schrödinger operators.
\begin{coro}
 Under the assumptions of the previous theorem, it follows that
 \begin{equation*}
  \che{f} + \| f 1_Y\|^2\ge  \frac{\varrho}{\lfloor\gamma\rfloor}\min\{\lambda,\frac{1}{10}\}\| f\|^2\mbox{  for all  }f\in L^2(X,\m).
 \end{equation*}
\end{coro}

\section{Application to quasi-regular Dirichlet forms}\label{section:Dirchlet}
We now turn to Dirichlet forms. We will freely use standard background from Dirichlet forms as can be found, e.g., in the textbooks \cite{Fuk,maroeckner} and follow \cite{dello,kuwae-98,kuwae} for the more technical aspects of quasi-regularity. Let us briefly introduce the latter concept. We always assume that $X$ is a separable, metric space endowed with a Radon measure $\m$ of full support. The metric is not needed for quasi-regularity, we use it for decomposing the space into balls.

In this section $(\En ,\F)$ denotes a symmetric Dirichlet form on $L^2(X,\m)$ with domain $\F$. It is called \emph{quasi-regular} provided the following conditions hold:
\begin{itemize}
 \item[(QR1)] There exists an $\En$--nest $(F_n)_{n\in\NN}$ consisting of compact sets, where  $\En$--nest means, that any $u\in\F$ can be approximated (in the form norm) by $u_n\in\F$ such that $u_n=0$ on $X\setminus F_n$.
 \item[(QR2)] There exists a dense subset $\F_0\subset \F$  every element of which has a quasi-continuous version.
 \item[(QR3)] There exists a polar set $N\subset X$ and a countable set of elements of $\F$ with quasi-continuous versions that separate the points of $X\setminus N$. 
\end{itemize}
\begin{remark}
 Quasi-regularity has been invented to deal with infinite dimensional situations and still retain most of the tools available for \emph{regular} Dirichlet forms. Regularity means that the underlying space is locally compact and there is a subset of $\F\cap C_c(X)$ that is dense both in $C_c(X)$ with respect to $\|\cdot\|_\infty$ and in $\F$ with repect to the energy norm.
\end{remark}
To construct the forms $\En_B$ for open balls, we next recall the Beurling--Deny--Le Jan representation of a Dirichlet form. We denote by $\mathfrak{M}^{\pm}_{\sigma}(Y,\cal{N}_\En)$ the set of extended $\sigma$-finite signed Borel measures on $Y$ that charge no $\En$-polar sets and by $\F_{loc}^\bullet$ the broad local space associated with $(\En,\F)$, see \cite{kuwae-98,dello} for details.
\begin{theorem*}[Theorems 5.1, 5.2, Lemma 5.1, 5.2  from \cite{kuwae-98} and Proposition 2.12 from \cite{dello}]
Let $(\En,\F)$ be a quasi--regular Dirichlet form. Then, the following holds: 
\\ (a)  Then there exist unique $\En^c, J, k$ satisfying
\begin{itemize}
\item[(i)] $\En^c$ is a positive definite symmetric bilinear form which is strongly local in
the general sense, i.e., if $u,v\in\F$ with $u = const$, $\m$-a.e. on a $\En$--neighbourhood of $\mathrm{supp}[v]$, then $\En^c(u,v)=0$.
\item[(ii)] $J$ is a $\sigma$-finite symmetric positive Borel measure on $X\times X\setminus\bigtriangleup$, where $\bigtriangleup$ is the diagonal, such that $J(N\times X)=0$ for any polar set $N$.
\item[(iii)] $k\in \mathfrak{M}^{\pm}_{\sigma}(X,\cal{N}_\En)$ is positive.
\item[(iv)] For any $u,v\in\F$:
$$
\En(u,v)=\En^c(u,v)+\int_{X\times X\setminus\bigtriangleup}\left(\tilde{u}(x)-\tilde{u}(y)\right)\left(\tilde{v}(x)-\tilde{v}(y)\right) dJ(x,y)+\int_X\tilde{u}(x)\tilde{v}(x) dk(x),
$$
where $\tilde{u}, \tilde{v}$ denote quasi-continuous versions of $u,v$, respectively.
\end{itemize}

(b)  There is a bilinear map,  
$$
\F_{loc}^\bullet\to \mathfrak{M}^{\pm}_{\sigma}(X,{\cal{N}}_\En), (f,g)\mapsto
\mu^c_{<f,g>},
$$
representing $\En^c$ in the sense that
$$
\En(f,g)=\mu^c_{<f,g>}(X) .
$$
\end{theorem*}
With these data at hand, we can define $\En_B$ for any open ball (actually any quasi-open set) $B\subset X$ by
\begin{align*}
H^1(B)& :=\{f\in L^2(B,\m) : \exists f^\#\in \F\mbox{  such that  }f^\#|_B=f\}\\
\En_B(f)& := \mu^c_{<f^\#, <f^\#>}(B)+ \int_{B\times B\setminus\bigtriangleup}\left(\tilde{f}(x)-\tilde{f}(y)\right)^2 dJ(x,y)+\int_B\tilde{f}^2(x) dk(x)
\end{align*}
Note that we used the somewhat unconventional notation $ H^1(B)$ for the active reflected space, as it is called in Dirichlet form theory (see, e.g.,  \cite{kuwae} for details) to stress the analogy with the classical situation.

\begin{theorem}
Let $\mathcal{E}$ be a quasi-regular Dirichlet form  on the  metric measure space $(X,\m,\dm)$ as above. Assume that $(X,\m,\dm)$ satisfies the volume comparison condition \eqref{VCR} with constant $\gamma$ and  $\En$ satisfies a local Poincar\'{e} inequality at scale $R$ with constant $\lambda$ in the sense that for every open ball $B=U(p,R)$ and $f\in H^1(B)$:
$$
\lambda\int_B|f-A_B (f)|^2d\m\le \En_B(f) .
$$
Then, for all $f\in L^2 (X,\m)$, 
$$
\En(f)\le \frac{\lambda\varrho}{\lfloor \gamma\rfloor}\| f\|^2
\quad\Longrightarrow\quad \| f\|^2\le\frac{10 \lfloor \gamma\rfloor}{\varrho}\| f 1_Y\|^2,
$$
whenever $Y\subset X$ is measurable and satisfies $\m(Y\cap B)\geq \varrho \m(B)$ for any ball  $B$ of  radius $R$.
\end{theorem}

The proof goes along the exact same lines as the proof of Theorem \ref{mms} above.

\begin{remark} The reader should think of a vanishing killing term, i.e., the case $k=0$ even though this is not formally necessary. The reason is that the presence of $k$ might cause the set of functions of low energy to be void.  On the other hand, in the generality of the Theorem above, it cannot be guaranteed that this set of functions is nonempty.
\end{remark}

\begin{remarks}

\begin{itemize}

\item[(a)] Regular Dirichlet forms arise in many situations. In particular, 
the  Laplace Beltrami operator on a Riemannian manifold comes from a regular Dirichlet form with  
$\mathcal{E}(f) = \int_X \nabla f d\mathrm{vol}$
So, the theorem applies to those classical situations, provided volume doubling and a suitable Poincar\'{e} inequality is valid. This in turn is guaranteed by lower bounds on the Ricci curvature. Volume doubling is a consequence of the Bishop-Gromov comparison theorem, see \cite{saloff}, Thm 5.6.4 for a proof and \cite{wei} for a survey that points at a number of generalizations.\\
In \cite{saloff}, Thm 5.6.5  one can also find a discussion of the Poincaré inequality, which is due to Buser, \cite{buser}.

\item[(b)] The setup from \cite{pesenson} is clearly more restrictive than our framework here.

\item[(c)] As a very special case, we can treat the setup from \cite{stolzman}. Clearly, the condition there is more incisive as ours here and the resulting estimate slightly weaker. At the same time, the much more complicated method of proof gives additional information that is not accessible by the approach here, namely lower bounds on Dirichlet Laplacians.

\item[(d)] We should also mention the recent \cite{BST} which equally well applies to general Schrödinger operators in Euclidean space. The uncertainty estimate obtained there works for the same class of operators and sets as in the present case, giving almost the same estimates. See the analogous discussion of the graph case in the subsequent section.

\item[(e)] Forms  with $\mathcal{E}(f) = \int_{X\times X\setminus \triangle (X)} (\widetilde(f) (x) -\widetilde{f}(y))^2 dJ (x,y)$
 arise in the study of graphs. So, the theorem applies to this case as well, provided volume doubling and Poincar\'{e} inequality is valid. In fact, we will provide a rather general result for graphs in the next section based on a completeness and a local compactness condition.
\item[(e)] Symmetry of the form is not essential; we could treat semi-Dirichlet forms in the same manner as above, partly replacing results from \cite{kuwae-98} by the non-symmetric analog in \cite{hu-ma-sun}.

\end{itemize}

\end{remarks}

\section{Application to graphs}\label{section-graph}
In this section we discuss an application of our method to graphs. Our method relies on a decomposition and a Poincar\'{e} inequality. The work \cite{LSS-17} provides a decomposition (Voronoi construction)  for suitable graphs and the work  \cite{LSS-18} provides a Poincar\'{e} inequality. Combining these two for - what we call -  $(R,\varrho)$-thick sets then allows us to apply our method do conclude a lower  bound. Details are given next.  We stress that - unlike most of the existing literature on graphs -  we do not assume any form of local summability condition on the vertex weight of the  graph (let alone a condition of local finiteness). The reason is that we do not need  our energy functionals  to be densely defined. 
This shows again the flexibility of our approach. 

\medskip

A \textit{weighted graph} is a triple  $(X,b,m)$ satisfying the following
properties:
\begin{itemize}
 \item  $X$ is an arbitrary  set, whose elements are  referred to as
\textit{vertices};
 \item $b:X\times X\to [0,\infty)$  is a symmetric  function with $b(x,x)=0$
for all $x\in
 X$.
 \item  $m:X\to (0,\infty)$ is  a  function on the vertices.
\end{itemize}
An element $(x,y)\in X\times X$ with $b(x,y)>0$ is then  called an
\textit{edge}  and $b$ is denoted as \textit{edge weight}; The
positive function $m:X\to (0,\infty)$ gives a measure on $X$ of
which we think of as a volume. In particular, we define
$$\vol (\Omega):=\sum_{x\in \Omega} m(x)$$
for $\Omega \subset X$.

Any graph comes with both spectral and geometric data (and the interplay between these two is a topic of much interest).

Spectral data is given by  the \textit{energy} $\mathcal{E}$
associated to the graph  defined by the quadratic form
$$
\mathcal{E}(f):=
\frac12 \sum_{x,y\in
X}b(x,y)(f(x)-f(y))^2 \mbox{  for  } f: X\longrightarrow \RR, 
$$
which may assume the value $\infty$ for the time being. 
The underlying Hilbert
space is
$$
\ell^2(X,m):=\{ f\in\RR^X\mid  \sum_{x\in X} f(x)^2m(x)<\infty\}
$$
with inner product and norm given by
$$\langle f,g\rangle =\sum_{x\in X} f(x) g(x) m(x) \mbox{ and }
\|f\| = \sqrt{\langle f, f\rangle}$$ respectively.

Geometric data comes from associating a metric to the graph.  Various metrics can (more or less naturally) be attached to a graph. Here, we rely on a metric coming from paths as follows: 
A sequence of vertices  $\gamma=(x_0, ..., x_k)$ is called a
\textit{path}  from $x$ to $y$ if $x_0 = x$, $x_k = y$ and
$b(x_l,x_{l+1}) >0$ for $l=0,\ldots, k-1$. 
The \textit{length}  $L(\gamma)$  of a path $\gamma$ is given by
$$
L(\gamma):=\sum_{j=0,...,k-1}\frac{1}{b(x_j,x_{j+1})} .
$$A natural distance to consider is the function
$$d : X\times X\longrightarrow [0,\infty] \mbox{ with } 
d(x,y):=\inf\{L(\gamma)\mid \gamma\mbox{  a path from }x\mbox{  to
}y\},
$$
for $x,y\in X$ with $x\neq y$ and 
$d(x,x) =0$ for all $x\in X$. Here, the infimum over the empty set is defined to be $\infty$. 
Then, $d$ is a pseudometric in the wide sense, i.e. it  is symmetric
and satisfies the triangle inequality. Clearly, in this generality,
$d$ need not separate the points of $X$.  A sufficient condition for $d$ to separate the points is that the graph  \textit{locally finite} (i.e. to any $x\in X$ the cardinality of $\{y\in X : b(x,y)>0\}$ is finite).

We denote by
$$ U_r(x):= \{y\in X \mid d(x,y) < r\} \;\mbox{ and }\; B_r(x):= \{y\in
X \mid d(x,y) \le r\}$$ the  open and closed balls of radius $r$,
respectively.

A graph is  \textit{connected} if any two different vertices admit a path  from one to the other. 
Hence, the graph  is connected if and only if $d$ takes only finite values.  
A connected graph is  called  \textit{geodesic} if for any $x,y\in X$ there
exists a path $\gamma$ from $x$ to $y$ with $L(\gamma) = d(x,y)$.

\begin{definition} Let $(X,b,m)$ be a connected  graph. A subset $D$ of $X$ is called $(R,\varrho)$-thick if  it satisfies both 
\begin{itemize}

\item $\bigcup_{x\in D} B_R (x) =X$ and  

\item $\vol(\{x\})   \geq \varrho \vol(B_R (x))$ for all $x\in D$. 

\end{itemize}

\end{definition}

\begin{theorem}[Uncertainty for graphs] \label{theorem-graph} Let $(X,b,m)$ be a  connected geodesic graph all of whose balls are finite sets. Let $\mathcal{E}$ be the energy associated to  $(X,b,m)$.  
Let  $D\subset X$ be  $(R,\varrho)$-thick and set 
$\vol(R):=\sup_{x\in D} \vol(B_R (x))$. 
Then, for any $f\in L^2(X,\m)$, 
\begin{equation}\tag{graph-UP}\label{GUP}
\mathcal{E} (f) \leq \frac{\varrho}{4\cdot  R \cdot \vol(R)} \|f\|^2 \Longrightarrow 
\|f\|^2 \leq \frac{10}{\varrho} \|f 1_D\|^2.
\end{equation}

\end{theorem}

The proof is based on our main abstract result and two auxiliary tools. These will be discussed next. The proof is then given at the end of the section.  We start with a discussion of the relevant  decompositions (see \cite{LSS-17} for further details and proofs).

\begin{definition}\label{def-vor}
Let $(X,b,m)$ be a graph and  $D\subset X$ non-empty. A
\emph{Voronoi decomposition} of $X$ with centers from $D$ is a
pairwise disjoint family $(V_p)_{p\in D}$ such that  following
conditions hold:
 \begin{enumerate}
  \item[(V1)] For each $p\in D$ the point $p$ belongs to  $ V_p$  and
  for all $x\in V_p$ there exists a path $\gamma$ from $p$ to $x$ that lies
   in $V_p$ and satisfies  $L(\gamma) = d(p,x)$.
  \item[(V2)] For each $p\in D$ and for all  $x\in V_p$ the inequality
  $d(p,x) \leq  d(q,x)$ holds  for any $q\in D$.
  \item[(V3)] $\bigcup_{p\in D}V_p=X$.
 \end{enumerate}
\end{definition}

Existence of Voronoi decompositions can be shown for geodesic graphs all of whose balls are finite sets. 

\begin{prop}[Existence of Voronoi decomposition, \cite{LSS-17}]\label{prop-vor}
Let $(X,b,m)$ be a connected geodesic  graph all of whose balls are finite sets.  Assume that  $D\subset X$ is
non-empty.  

(a) Then there exists a Voronoi decomposition with centers
from $D$. 

(b) Whenever $0<R <\infty$ satisfies $\bigcup_{x \in D} B_R (x) =X$ then any 
Voronoi decomposition $(V_p)_{p\in D}$  of $X$ with centers from $D$
has the property that $V_p\subset B_R(p)$ for all  $p\in D$.
\end{prop}

We note that the finiteness conditions on balls from the previous proposition implies that the set $X$ is countable.

We now turn to the relevant Poincar\'{e} inequality. Let $F$ be  a non-empty subset of $X$ with 
$$\vol (F)<\infty$$
  and denote by $m_F$ the restriction of $m$ to $F$.   Then, $(F,m_F)$ is a finite  measure space. Denote the inner product on 
 $\ell^2 (F,m_F)$ by $\langle \cdot,\cdot\rangle_F$ and the corresponding norm by $\|\cdot\|_F$. 
The restriction  $\en_F$ of $\en$  to  $F$ is the energy functional on $\ell^2 (F,m_F)$ given  by
$$\en_F (f) :=\frac{1}{2} \sum_{x,y\in F} b(x,y) (f(x) - f(y))^2.$$
We define 
$$\lambda_1^F :=\inf\left\{\frac{\en_F (f)}{\|f\|} :0\neq  f\in \ell^2 (F,m) \mbox{ with }  \langle f, \underline{1}\rangle_F = 0\right\}.$$
Then, a simple argument gives 
$$ \lambda_1^F \|f - A(f) \underline{1}\|_F \leq \en_F (f)$$
  for all $f\in \ell^2 (F,m_F)$ (see remark in Section \ref{section:technical-core} as well). 

(We note in passing that this includes the case that $F$ consists of only one point. In this case $\|f- A(f) \underline{1}\|_F =0$, $\en_F =0$ and $\lambda_1^F = \infty$ as it is the infimum over an empty set  and we use the convention that $\infty \cdot 0 = 0$.)

We are going to provide a lower bound on $\lambda_1^F$ in geometric terms. The \textit{intrinsic diameter} of a subset $\varOmega$ of  $X$ is given by
$$\diam (\varOmega) :=\sup_{x,y\in \varOmega}\inf\{ L(\gamma) : \mbox{ $\gamma$ path from $x$ to $y$ in $\varOmega$}\}.$$
We note that the intrinsic diameter is defined with respect to  (paths in)  $\varOmega$ only. So,  $\diam (\varOmega)$ may be strictly bigger that $\sup_{x,y\in\varOmega} d(x,y)$. In particular, $\diam (\varOmega) = \infty$ if there are $x,y\in \varOmega$ that are not joined by a path in $\varOmega$. 

The following lemma is well known, see e.g.  the survey  \cite{LS-20},  to which we refer for further references as well.

\begin{lemma}[Universal Poincar\'{e}] Let $(X,b,m)$ as above and let $F$ be a subset of $X$ with $\vol (F)<\infty$. Then, 
$$
\lambda_1^F\ge\frac{4}{\diam(F)\vol(F)}
$$
holds. Here, $\diam (F) =\infty$ is possible in which case the right hand side of the inequality is zero. 
\end{lemma}

After these preparations we can now provide the proof of the main result in this section.

\begin{proof}[Proof of Theorem \ref{theorem-graph}]
By Proposition \ref{prop-vor} there exists a Voronoi decomposition $V_p$, $p\in D$. In particular, for any $p\in D$ the set   $V_p$ contains $p$  and  each element of $V_p$ allows for a path to $p$ completely contained in $V_p$ and  of length less than $R$.  In particular, $V_p$ is contained in $B_R (p)$ and has diameter less than $2 R$. By finiteness of balls this implies also that $V_p$ is finite and, hence, has finite volume. 
Given this,  the universal Poincar\'{e} inequality gives
$$
\lambda_1^{V_p} \ge\frac{4}{\diam(V_p)\vol(V_p)} \geq \frac{2}{ R \vol(R)}$$
for each $p\in D$. Moreover, for any $f$ on $X$ we clearly have 
$$ \sum_{p\in D} \mathcal{E}_{V_p} (f|_{V_p}) \leq \en (f).$$
Together with the fact that $D$ is $(R,\varrho)$-thick, this  implies that $(V_p)$, $p \in D$, provide a $(\varrho,1)$-decomposition. So, the statement follows from the main abstract result. 
\end{proof}

We now compare the results here with those in \cite{LSS-17} where a completely different method of proof was employed.
\begin{itemize}
 \item Our estimate \eqref{GUP} gives an uncertainty estimate for low energy functions, where the energy cut-off 
 $$
 E=\frac{\varrho}{4\cdot  R \cdot \vol(R)}
 $$
 depends on the density $\varrho$ of the $(R,\varrho)$-thick set $D$, see above.
 \item We do not require the graph Laplacian to be bounded. 
 \item In terms of uncertainty estimates, we get
 \begin{equation}
 \label{eq:3}
\| 1_D f\|^2 \geq \frac{\varrho}{10} \| f\|^2\mbox{   if  }f\in Ran(P_E) .
\end{equation}
 with $E$ as above. 
 \item With the method from \cite{LSS-17}, we can treat $D$ which are $R$-dense in the sense that every point in $X$ has distance at most $R$ from $D$; this property is clearly less restrictive than $(R,\varrho)$-thickness. We then get estimates like \eqref{eq:intro2} with a somewhat complicated $\kappa$ for any 
 $$E<\frac{1}{R\cdot \vol(R)}.$$
 So the energy range is somewhat better, but the constants somewhat worse and the estimate holds for bounded Laplacians only.
\end{itemize}
In \cite{BST} it was mentioned that the method can be extended to graphs. We do not go into details here, just roughly sketch the way of the argument. We start with a graph like in the above Theorem with the additional requirement of local boundedness, so that we have a densely defined graph Laplacian $H$ at our disposal. With a Voronoi decomposition as above we obtain, for a  $(R,\varrho)$-thick set $D$ that
\begin{equation}\label{LB}\tag{LB}
H+1_D\ge \frac{\varrho}{1+4\cdot R\cdot\vol(R)}.
\end{equation}

The argument runs along the lines of the proof of Corollary 4.2 from \cite{BST}, where the Voronoi decomposition replaces the decomposition into cubes and the universal Poincaré inequality gives the necessary uniform bound. As readily seen, this latter estimate \eqref{LB} gives an uncertainty estimate of the type \eqref{GUP}, more specifically, for any $E<\frac{\varrho}{1+4\cdot R\cdot\vol(R)}$,
\begin{equation}\label{bst}
 \mathcal{E} (f) \leq E \|f\|^2 \Longrightarrow 
\|f\|^2 \leq  \left(\frac{\varrho}{1+4\cdot R\cdot\vol(R)}-E\right)^{-1}\|f 1_D\|^2.
\end{equation}
We note that this uncertainty principle is slightly weaker than what we obtain with our new method here, and , as mentioned above, it needs more restrictive conditions.

In \cite{LSS-17} we discussed estimates from earlier papers, in particular estimates that use the isoperimetric or Cheeger constant, as does \cite{pesenson-2}. For the local estimate on finite parts, there is not too much difference, so we do not repeat this discussion here. On the other hand, for the infinite case, our method here is the first that applies to unbounded Laplacians. In fact, since we are not restricted to quadratic energy forms, much more general classes of energy functionals might be treated by our method.


\begin{thebibliography}{10}

\bibitem{al}
A.~Agaskar and Y.~M. Lu.
\newblock A spectral graph uncertainty principle.
\newblock {\em IEEE Transactions on Information Theory}, 59(7):4338--4356,
  2013.

\bibitem{ags-2005}
L.~Ambrosio, N.~Gigli, and G.~Savaré.
\newblock {\em Gradient flows: in metric spaces and in the space of probability
  measures}.
\newblock Springer Science \& Business Media, 2005.

\bibitem{ags-2014}
L.~Ambrosio, N.~Gigli, and G.~Savaré.
\newblock Calculus and heat flow in metric measure spaces and applications to
  spaces with ricci bounds from below.
\newblock {\em Inventiones Mathematicae}, 195(2):289--391, 2014.

\bibitem{ago}
A.~Anis, A.~Gadde, and A.~Ortega.
\newblock Towards a sampling theorem for signals on arbitrary graphs.
\newblock In {\em 2014 IEEE International Conference on Acoustics, Speech and
  Signal Processing (ICASSP)}. IEEE, 2014.

\bibitem{BeauchardP-18}
K.~Beauchard and K.~Pravda-Starov.
\newblock Null-controllability of hypoelliptic quadratic differential
  equations.
\newblock {\em J. \'Ec. polytech. Math.}, 5:1--43, 2018.

\bibitem{BST}
A.~B{\"o}ttcher, P.~Stollmann, and M.~Tautenhahn.
\newblock Uncertainty principles and lower bounds for {S}chr{\"o}dinger
  operators.
\newblock {\em Pure and Applied Functional Analysis, to appear}, 2026.
\newblock arXiv:2606.15424.

\bibitem{BLS}
A.~Boutet~de Monvel, D.~Lenz, and P.~Stollmann.
\newblock An uncertainty principle, {W}egner estimates and localization near
  fluctuation boundaries.
\newblock {\em Mathematische Zeitschrift}, 269(3):663--670, 2011.

\bibitem{buser}
P.~Buser.
\newblock A note on the isoperimetric constant.
\newblock {\em Ann. Sci. Ecole Norm. Sup.}, 15:213--230, 1982.

\bibitem{cheeger}
J.~Cheeger.
\newblock Differentiability of {L}ipschitz functions on metric measure spaces.
\newblock {\em Geometric \& Functional Analysis (GAFA)}, 9:428--517, 1999.

\bibitem{cvsk}
S.~Chen, R.~Varma, A.~Sandryhaila, and J.~Kova{\v c}evi{\'c}.
\newblock Discrete signal processing on graphs: Sampling theory.
\newblock {\em IEEE Transactions on Signal Processing}, 63(24):6510--6523,
  2015.

\bibitem{CHK-03}
J.-M. Combes, P.~D. Hislop, and F.~Klopp.
\newblock H{\"o}lder continuity of the integrated density of states for some
  random operators at all energies.
\newblock {\em International Mathematics Research Notices}, 2003(4):179--209,
  2003.

\bibitem{CHK-07}
J.-M. Combes, P.~D. Hislop, and F.~Klopp.
\newblock An optimal {W}egner estimate and its application to the global
  continuity of the integrated density of states for random {S}chr{\"o}dinger
  operators.
\newblock {\em Duke Mathematical Journal}, 140(3):469--498, 2007.

\bibitem{dello}
L.~Dello~Schiavo and K.~Suzuki.
\newblock Rademacher-type theorems and {S}obolev-to-{L}ipschitz properties for
  strongly local {D}irichlet spaces.
\newblock {\em Journal of Functional Analysis}, 281(11):109234, 2021.

\bibitem{Douglas-66}
R.~G. Douglas.
\newblock On majorization, factorization, and range inclusion of operators on
  {H}ilbert space.
\newblock {\em Proceedings of the American Mathematical Society},
  2(17):413--415, 1966.

\bibitem{EgidiGST-24}
M.~Egidi, D.~Gallaun, C.~Seifert, and M.~Tautenhahn.
\newblock Sufficient criteria for stabilization properties in {B}anach spaces.
\newblock {\em Integral Equations and Operator Theory}, 96(2), 2024.
\newblock Article number 13.

\bibitem{Fuk}
M.~Fukushima, Y.~{\=O}shima, and M.~Takeda.
\newblock {\em {D}irichlet forms and symmetric Markov processes}, volume~19 of
  {\em de Gruyter Studies in Mathematics}.
\newblock Walter de Gruyter, Berlin, 1994.

\bibitem{GallaunST-20}
D.~Gallaun, C.~Seifert, and M.~Tautenhahn.
\newblock Sufficient criteria and sharp geometric conditions for observability
  in {B}anach spaces.
\newblock {\em SIAM Journal on Control and Optimization}, 58(4):2639--2657,
  2020.

\bibitem{heinonen}
Juha Heinonen.
\newblock {\em Lectures on Analysis on Metric Spaces}.
\newblock Springer Science \& Business Media, 2001.

\bibitem{hu-ma-sun}
Z.~C. Hu, Z.~M. Ma, and W.~Sun.
\newblock Extensions of {L}{\'e}vy--{K}hintchine formula and {B}eurling--{D}eny
  formula in semi-{D}irichlet forms setting.
\newblock {\em Journal of Functional Analysis}, 239(1):179--213, 2006.

\bibitem{JerisonL-99}
D.~Jerison and G.~Lebeau.
\newblock Nodal sets of sums of eigenfunctions.
\newblock In M.~Christ, C.~E. Kenig, and C.~Sadosky, editors, {\em Harmonic
  Analysis and Partial Differential Equations}, pages 223--239. The University
  of Chicago Press, Chicago, 1999.

\bibitem{kato}
Tosio Kato.
\newblock {\em Perturbation Theory for Linear Operators}, volume 132.
\newblock Springer Science \& Business Media, 2013.

\bibitem{K-13}
A.~Klein.
\newblock Unique continuation principle for spectral projections of
  {S}chr{\"o}dinger operators and optimal {W}egner estimates for non-ergodic
  random {S}chr{\"o}dinger operators.
\newblock {\em Communications in Mathematical Physics}, 323(3):1229--1246,
  2013.

\bibitem{kuwae-98}
K.~Kuwae.
\newblock Functional calculus for {D}irichlet forms.
\newblock {\em Osaka Journal of Mathematics}, 35(3):683--715, 1998.

\bibitem{kuwae}
K.~Kuwae.
\newblock Reflected {D}irichlet forms and the uniqueness of {S}ilverstein's
  extension.
\newblock {\em Potential Analysis}, 16:221--247, 2002.

\bibitem{LebeauR-95}
G.~Lebeau and L.~Robbiano.
\newblock Contr{\^o}le exact de l'{\'e}quation de la chaleur.
\newblock {\em Communications in Partial Differential Equations},
  20(1--2):335--356, 1995.

\bibitem{LebeauZ-98}
G.~Lebeau and E.~Zuazua.
\newblock Null-controllability of a system of linear thermoelasticity.
\newblock {\em Archive for Rational Mechanics and Analysis}, 141(4):297--329,
  1998.

\bibitem{LSS-18}
D.~Lenz, M.~Schmidt, and P.~Stollmann.
\newblock Topological {P}oincar{\'e} type inequalities and lower bounds on the
  infimum of the spectrum for graphs, 2018.
\newblock Preprint.

\bibitem{LS-20}
D.~Lenz and P.~Stollmann.
\newblock Universal lower bounds for {L}aplacians on weighted graphs.
\newblock In {\em Analysis and Geometry on Graphs and Manifolds}, volume 461 of
  {\em London Mathematical Society Lecture Note Series}, pages 156--171.
  Cambridge University Press, Cambridge, 2020.

\bibitem{LSS-17}
D.~Lenz, P.~Stollmann, and G.~Stolz.
\newblock An uncertainty principle and lower bounds for the {D}irichlet
  {L}aplacian on graphs.
\newblock {\em Journal of Spectral Theory}, 10(1):115--145, 2020.

\bibitem{maroeckner}
Zhi-Ming Ma and Michael R{\"o}ckner.
\newblock {\em Introduction to the Theory of (Non-Symmetric) {D}irichlet
  Forms}.
\newblock Springer Science \& Business Media, 2012.

\bibitem{Miller-10}
L.~Miller.
\newblock A direct {L}ebeau--{R}obbiano strategy for the observability of
  heat-like semigroups.
\newblock {\em Discrete and Continuous Dynamical Systems. Series B},
  14(4):1465--1485, 2010.

\bibitem{MRM-22}
Peter M{\"u}ller and Constanza Rojas-Molina.
\newblock Localisation for {D}elone operators via {B}ernoulli randomisation.
\newblock {\em Journal d'Analyse Math{\'e}matique}, 147(1):297--331, 2022.

\bibitem{MuenchSST}
F.~M{\"u}nch, C.~Seifert, P.~Stollmann, and M.~Tautenhahn.
\newblock On controllability, observability and stabilizability of the heat
  equation on discrete graphs.
\newblock {\em arXiv}, 2026.

\bibitem{NTTV-18}
I.~Naki{\'c}, M.~T{\"a}ufer, M.~Tautenhahn, and I.~Veseli{\'c}.
\newblock Scale-free unique continuation principle for spectral projectors,
  eigenvalue-lifting and {W}egner estimates for random {S}chr{\"o}dinger
  operators.
\newblock {\em Analysis and PDE}, 11(4):1049--1081, 2018.

\bibitem{NakicTTV-20}
I.~Naki{\'c}, M.~T{\"a}ufer, M.~Tautenhahn, and I.~Veseli{\'c}.
\newblock Sharp estimates and homogenization of the control cost of the heat
  equation on large domains.
\newblock {\em ESAIM: Control, Optimisation and Calculus of Variations},
  26(54):26, 2020.

\bibitem{pesenson-2}
I.~Z. Pesenson.
\newblock Sampling in {P}aley-{W}iener spaces on combinatorial graphs.
\newblock {\em Transactions of the American Mathematical Society},
  360:5603--5627, 2008.

\bibitem{pesenson}
I.~Z. Pesenson.
\newblock Sampling by averages and average splines on {D}irichlet spaces and on
  combinatorial graphs.
\newblock In M.~Hirn, S.~Li, K.~A. Okoudjou, S.~Saliani, and {\"O}.~Yilmaz,
  editors, {\em Excursions in Harmonic Analysis, Volume 6}, Applied and
  Numerical Harmonic Analysis. Birkh{\"a}user, Cham, 2021.

\bibitem{ptgv}
G.~Puy, N.~Tremblay, R.~Gribonval, and P.~Vandergheynst.
\newblock Random sampling of bandlimited signals on graphs.
\newblock {\em Applied and Computational Harmonic Analysis}, 44(2):446--475,
  2018.

\bibitem{reed-simon}
Michael Reed and Barry Simon.
\newblock {\em Methods of Modern Mathematical Physics: Functional Analysis},
  volume~1.
\newblock Gulf Professional Publishing, 1980.

\bibitem{RMV-13}
C.~Rojas-Molina and I.~Veseli{\'c}.
\newblock Scale-free unique continuation estimates and applications to random
  {S}chr{\"o}dinger operators.
\newblock {\em Communications in Mathematical Physics}, 320(1):245--274, 2013.

\bibitem{saloff}
Laurent Saloff-Coste.
\newblock {\em Aspects of {S}obolev-type Inequalities}, volume 289.
\newblock Cambridge University Press, 2002.

\bibitem{snfov}
D.~I. Shuman, S.~K. Narang, P.~Frossard, A.~Ortega, and P.~Vandergheynst.
\newblock The emerging field of signal processing on graphs: Extending
  high-dimensional data analysis to networks and other irregular domains.
\newblock {\em IEEE Signal Processing Magazine}, 30(3):83--98, May 2013.

\bibitem{PMFO}
P.~Stollmann.
\newblock A new uncertainty principle at low energies.
\newblock In David Damanik, Matthias Keller, Tatiana Nagnibeda, and Felix
  Pogorzelski, editors, {\em Geometry, Dynamics and Spectrum of Operators on
  Discrete Spaces}, volume~18 of {\em Oberwolfach Reports}, pages 33--85. 2021.

\bibitem{stolzman}
P.~Stollmann and Gunter Stolz.
\newblock Lower bounds for {D}irichlet {L}aplacians and uncertainty principles.
\newblock {\em Journal of the European Mathematical Society (JEMS)},
  23(7):2337--2360, 2021.

\bibitem{TenenbaumT-11}
G.~Tenenbaum and M.~Tucsnak.
\newblock On the null-controllability of diffusion equations.
\newblock {\em ESAIM: Control, Optimisation and Calculus of Variations},
  17(4):1088--1100, 2011.

\bibitem{tbl}
M.~Tsitsvero, S.~Barbarossa, and P.~Di~Lorenzo.
\newblock Signals on graphs: Uncertainty principle and sampling.
\newblock {\em IEEE Transactions on Signal Processing}, 64(18):4845--4860,
  2016.

\bibitem{WangZ-17}
G.~Wang and C.~Zhang.
\newblock Observability inequalities from measurable sets for some abstract
  evolution equations.
\newblock {\em SIAM Journal on Control and Optimization}, 55(3):1862--1886,
  2017.

\bibitem{wei}
G.~Wei.
\newblock Volume comparison and its generalizations.
\newblock {\em Advanced Lectures in Mathematics}, 22:311--322, 2012.

\end{thebibliography}
\end{document}